\documentclass[oneside,english]{amsart}
\usepackage[T1]{fontenc}
\usepackage[latin9]{inputenc}
\usepackage{amstext}
\usepackage{amsthm}
\usepackage{amssymb}
\usepackage{todonotes}
\usepackage{enumitem}
\usepackage{stmaryrd}
\usepackage{url}
\makeatletter
\numberwithin{equation}{section}
\numberwithin{figure}{section}
\theoremstyle{plain}
\newtheorem{thm}{\protect\theoremname}
\theoremstyle{plain}
\newtheorem{prop}[thm]{Proposition}
\newtheorem{question}[thm]{Question}

\newtheorem{lem}[thm]{\protect\lemmaname}
\newtheorem{claim}{\protect\theoremname}
\newtheorem{defn}[thm]{Definition}
\newtheorem{cor}[thm]{Corollary}
\makeatother

\usepackage{babel}
\providecommand{\lemmaname}{Lemma}
\providecommand{\theoremname}{Theorem}
\providecommand{\theoremname}{Claim}
\providecommand{\theoremname}{Preposition}

\providecommand{\theoremname}{Definition}

\newcommand{\ohad}[1]{\todo[size=\tiny, color=pink]{Ohad: #1}}

\title{The Forbidden Cross Intersection Problem for Permutations}

\author[N. Keller]{Nathan Keller}
\address{Department of Mathematics, Bar Ilan University, Ramat Gan, Israel}
\email{nkeller@math.biu.ac.il}
\thanks{The work of the first author is partially supported by the Israel Science Foundation (grants no.~2669/21 and 2456/25) and by the US-Israel Binational Science Foundation (grant no.~2024120). The work of the second author is partially supported by the European Research Council (StG no.~101163794), by the Israel Science Foundation (grant no.~1980/22), and by the US-Israel Binational Science Foundation (grant no.~2024120).}

\author[N. Lifshitz]{Noam Lifshitz}
\address{Einstein Institute of Mathematics, Hebrew University, Jerusalem, Israel}
\email{noamlifshitz@gmail.com}

\author[O. Sheinfeld]{Ohad Sheinfeld}
\address{Einstein Institute of Mathematics, Hebrew University, Jerusalem, Israel}
\email{oshenfeld@gmail.com}

\begin{document}

\maketitle

\begin{abstract}
    We prove the following, for a universal constant $c>0$. Let $n \in \mathbb{N}$ and $1 \leq t<c\frac{n}{\log n}$. Let $F,G \subset S_n$ be families of permutations such that no $\sigma \in F$ and $\tau \in G$ agree on exactly $t-1$ values. Then $|F||G| \leq ((n-t)!)^2$, with equality if and only if $F=G=\{\sigma \in S_n:\sigma(i_1)=j_1,\ldots,\sigma(i_t)=j_t\}$, for some $i_1,\ldots,i_t,j_1,\ldots,j_t \in \{1,2,\ldots,n\}$. 
    The range of values of $t$ in the result is essentially optimal, as for any $\epsilon>0$, the statement fails for $t=(1+\epsilon)\frac{n}{\log_2 n}$ and all $n>n_0(\epsilon)$. This solves the cross-intersection variant of the Erd\H{o}s-S\'{o}s forbidden intersection problem for permutations. The best previously known result, by Kupavskii and Zakharov (Adv.~Math., 2024), obtained the same assertion for $t \leq \tilde{O}(n^{1/3})$. We obtain our result by combining two recently introduced techniques: hypercontractivity of global functions and spreadness.
\end{abstract}

\section{Introduction}

\subsection{Previous work}

A family $F$ of subsets of $[n]=\{1,2,\ldots,n\}$ is called \emph{$t$-intersecting} if for any $A,B \in F$, we have $|A \cap B| \geq t$, and is called \emph{$(t-1)$-intersection free} if for any $A,B \in F$, we have 
$|A \cap B| \neq t-1$. Similarly, a pair of families $F,G$ is called \emph{cross $t$-intersecting} (resp., \emph{cross $(t-1)$-intersection free}) if $|A \cap B| \geq t$ (resp., $|A \cap B| \neq t-1$) for all $A \in F, B \in G$. 

\medskip \noindent \textbf{The $t$-intersection problem and the forbidden $(t-1)$-intersection problem for families of sets.} The problems of determining the maximum size of $t$-intersecting and $(t-1)$-intersection free families of $k$-element subsets of $[n]$ are among the best known problems in extremal combinatorics. The \emph{$t$-intersection problem} was raised in the 1961 paper of Erd\H{o}s, Ko, and Rado~\cite{EKR61} that initiated the research area of \emph{intersection theorems} (see~\cite{FT16}). Following partial results of Frankl~\cite{F78} and Wilson~\cite{Wilson84}, it was fully solved in 1997 in the celebrated \emph{complete intersection theorem} of Ahlswede and Khachatrian~\cite{AK97}. The theorem asserts that for each triple $(n,k,t)$, one of the \emph{Frankl families} $F_{n,k,t,r}=\{S \subset \binom{[n]}{k}:|S\cap [t+2r]| \geq t+r\}$ has the maximum size. 
The \emph{forbidden $(t-1)$-intersection  problem} turned out to be significantly harder. Raised in 1971 by Erd\H{o}s for $t=2$ and a few years later by Erd\H{o}s and S\'{o}s~\cite{Erd76} for a general $t$, this problem was studied in numerous works, and yet, a solution was obtained only for certain ranges of $(n,k,t)$. 

\medskip \noindent \textbf{Previous results on the forbidden $(t-1)$-intersection problem for families of sets.} For constant values of $t$, Erd\H{o}s and S\'{o}s conjectured that for sufficiently large $n$, the maximum size of a $(t-1)$-intersection free family $F \subset \binom{[n]}{k}$ is $\binom{n-t}{k-t}$, attained by the $t$-intersecting family $F=\{S \in \binom{[n]}{k}:[t] \subset S\}$. This conjecture was proved for all $k \geq 2t,n>n_0(k,t)$ in an influential paper of Frankl and F\"{u}redi~\cite{frankl1985forbidding}. Ellis et al.~\cite{EKL16} (basing upon~\cite{NN21}) proved it for all $(n,k,t)$ such that $k \geq 2t, n \geq n_0(t)$, and $n \geq (k-t+1)(t+1)$, which is the entire range of values of $k$ in which the maximum size of a $t$-intersecting family is  
$\binom{n-t}{k-t}$. They also showed that the maximum size of a $(t-1)$-intersection free family is equal to the maximum size of a $t$-intersecting family for 
all $2t \leq k \leq (\frac{1}{2}-\zeta)n, n>n_0(t,\zeta)$. Very recently, Kupavskii and Zakharov~\cite{KZ22} were the first to prove the Erd\H{o}s-S\'{o}s conjecture also in a setting where $t$ grows to infinity with $n$. They showed that it holds for $n = \lceil k^\alpha \rceil, t = \lceil k^\beta \rceil$, where    $0<\beta<0.5 $, $\alpha > 1+2\beta $, and $k\ge k_0(\alpha,\beta)$.

Another extensively studied range is $\epsilon n < t \leq (\frac{1}{2}-\epsilon)n$. In the 70's, Erd\H{o}s offered 250\$ for proving that in this range, the size of any $(t-1)$-intersection free family is less than $(2-\delta)^n$, where $\delta=\delta(\epsilon)$. This was proved in a strong form in the celebrated Frankl-R\"{o}dl~\cite{FranklR87} theorem, which led to significant applications in different areas, including discrete geometry~\cite{FranklR90}, communication complexity~\cite{Sgall99} and quantum computing~\cite{BuhrmanCW98}. Obtaining exact bounds for all $(n,k,t)$ looks completely elusive as for now. 

\smallskip \noindent \textbf{Families of permutations.} In the last decades, the $t$-intersection problem and the forbidden intersection problem were studied in 
various other settings -- e.g., for graphs~\cite{EFF12}, set partitions~\cite{MM05}, linear maps~\cite{Linear-maps}, and codes~\cite{keevash2023forbidden}. Arguably, the most thoroughly studied setting (except for $k$-subsets of $[n]$) is \emph{families of permutations}. A family $F \subset S_n$ is $t$-intersecting (resp., $(t-1)$-intersection free) if for any $\sigma,\tau \in F$, $|\{i \in [n]:\sigma(i)=\tau(i)\}|\geq t$ (resp., 
$|\{i \in [n]:\sigma(i)=\tau(i)\}|\neq t-1$). The notions of cross $t$-intersecting and cross $(t-1)$-intersection free pairs of families of permutations are defined similarly. 

In 1977, Deza and Frankl~\cite{DF77} asked whether an analogue of the  Erd\H{o}s-Ko-Rado theorem holds for families of permutations -- i.e., whether for any $n \geq n_0(t)$ the maximum size of a $t$-intersecting family of permutations is $(n-t)!$ (which is attained by the family $\{\sigma \in S_n: \forall i \in [t], \sigma(i)=i\}$). The question turned out to be much harder than the $t$-intersection problem for subsets of $[n]$. Deza and Frankl managed to prove their conjecture only for $t=1$, and the case $t>1$ remained open for over 30 years until it was resolved by Ellis, Friedgut and Pilpel~\cite{EFP11} in 2011. Ellis et al.~proved the  conjecture for all $n > n_0(t)$ (not specifying the value of $n_0$) and conjectured that an analogue of the complete intersection theorem holds for families of permutations for all $(n,t)$. Despite significant progress in recent years (including~\cite{DN22,keller2024t,KZ22}), this conjecture is still open. The best known result, by Kupavskii~\cite{Kup24a}, proves the conjecture of Ellis et al.~\cite{EFP11} for all $t<(1-\epsilon)n$ and $n>n_0(\epsilon)$.  

\smallskip \noindent \textbf{Previous results on the forbidden $(t-1)$-intersection problem for families of permutations.} In 2014, Ellis~\cite{Ell14} initiated the study of the forbidden intersection problem for permutations. He showed that the maximum size of a $1$-intersection free family of permutations is $(n-2)!$, and conjectured that for all $n \geq n_0(t)$, the maximum size of a $(t-1)$-intersection free family $F \subset S_n$ is $(n-t)!$. 
Ku et al.~\cite{KuLW17} proved the weaker upper bound $(n-\frac{t}{2})!$ for a sufficiently large $n$. Both works used eigenvalue techniques and the representation theory of $S_n$. Ellis and Lifshitz~\cite{DN22} showed that the conjecture of Ellis holds for all $t=O(\frac{\log n}{\log\log n})$, using the discrete Fourier-analytic \emph{junta method}~\cite{NN21}, along with a representation-theoretic argument. 
The best known result on the problem was obtained in a recent work of Kupavskii and Zakharov~\cite{KZ22} who proved
the conjecture for all $t=\tilde{O}(n^{1/3})$.\footnote{In~\cite{keller2024t}, the authors claimed that the techniques they used for the $t$-intersection problem in $S_n$ can be used to prove the conjecture of Ellis on the forbidden $(t-1)$-intersection problem for all $t \leq \tilde{O}(n^{1/2})$. However, no proof was provided.} Unlike previous works in this direction, the work of Kupavskii and Zakharov does not use any specific properties of $S_n$ (and in particular, its representation theory). Instead, they prove that the  conjecture holds for the general setting of $(t-1)$-intersection free families in $\mathcal{U}$, where the `universe' $\mathcal{U}$ is a `pseudorandom' sub-family of ${N}\choose{M}$, and show that $S_n$ can be viewed as a `pseudorandom' sub-family of ${n^2}\choose{n}$. The technique of Kupavskii and Zakharov, called the \emph{spread approximation method}, builds on the techniques developed in the breakthrough result of Alweiss et al.~\cite{ALWZ21} on the Sunflower Lemma. This method was further developed in a series of very recent papers and was used to obtain numerous applications to intersection theorems in various settings (see, e.g.,~\cite{FranklK25,Kupavskii23partitions,Kupavskii23hereditary,KupavskiiN24}). 

In the range $\epsilon n < t \leq (1-\epsilon)n$, a Frankl-R\"{o}dl-type theorem was obtained by Keevash and Long~\cite{KL17}. They showed that any $(t-1)$-intersection free family $F \subset S_n$ satisfies $|F| \leq (n!)^{1-\delta}$, where $\delta=\delta(\epsilon)$.

\medskip \noindent \textbf{The cross-intersection setting.}
In many of the works on the $t$-intersection problem and on the forbidden intersection problem, both for sets and for permutations,  the assertion is obtained via a seemingly stronger result -- a sharp upper bound on $|F||G|$, where $F,G$ are cross $t$-intersecting and cross $(t-1)$-intersection free families, respectively (see, e.g.,~\cite{Ell14,EKL16,DN22,FranklR87,NN21,keller2024t}). Hence, the solution of the cross-intersection variant of the problem is interesting not only for its own sake, but also as an important step toward the corresponding intersection problem. For more details on works in the cross-intersection setting, the reader is referred to~\cite[Section~10]{FT16}.

\subsection{Our results}
In this paper we consider the cross-intersection variant of the forbidden intersection problem in $S_n$. In this setting, we prove an essentially optimal variant of the aforementioned conjecture of Ellis~\cite{Ell14}.
\begin{thm}\label{thm:main}
    There exists $c>0$ such that the following holds.
    Let $n,t \in \mathbb{N}$ such that $t\leq c \frac{n}{\log{n}}$. Let $F,G \subset S_n$ such that for any $\sigma \in F, \tau \in G$, $|\{i \in [n]:\sigma(i)=\tau(i)\}|\neq t-1$. Then 
    \[
    |F||G| \leq (n-t)!^2.
    \]
\end{thm}
The range of $t$ in the theorem is essentially optimal, as for any $\epsilon>0$, the assertion of the theorem fails for all $t \geq (1+\epsilon)\frac{n}{\log_2 n}$ and $n \geq n_0(\epsilon)$. Indeed, for any even $n$, the families 
\begin{align*}
&F=\left\{\sigma \in S_n:\sigma \left( \left[\frac{n}{2}\right]\right)=\left[\frac{n}{2} \right], \sigma \left([n] \setminus \left[\frac{n}{2}\right]\right)=[n] \setminus \left[\frac{n}{2}\right]\right\}, \\
&G=\left\{\tau \in S_n:\tau \left(\left[\frac{n}{2}\right] \right)=[n] \setminus \left[\frac{n}{2}\right], \tau \left([n] \setminus \left[\frac{n}{2}\right]\right)=\left[\frac{n}{2}\right]\right\}
\end{align*}
are clearly cross $(t-1)$-intersection free. We have $|F||G|= (\frac{n}{2})!^4$, which for any $\epsilon>0$ is larger than $((n-t)!)^2$ for all $t \geq (1+\epsilon)\frac{n}{\log_2 n}$, provided $n \geq n_0(\epsilon)$.

\medskip \noindent \textbf{Stability version.} Our proof of Theorem~\ref{thm:main} implies the following stability version.
\begin{thm}\label{thm:stability}
  There exists $c>0$ such that the following holds. Let $n,t \in \mathbb{N}$, such that $t\leq c \frac{n}{\log{n}}$. Let $F,G \subset S_n$ such that for any $\sigma \in F, \tau \in G$, $|\{i \in [n]:\sigma(i)=\tau(i)\}|\neq t-1$. If 
    \[
    |F||G| >  \left(1-\frac{1}{(100t)^t} \right)\cdot (n-t)!^2,
    \]
    then $F,G \subset \{\sigma \in S_n:\sigma(i_1)=j_1,\ldots,\sigma(i_t)=j_t\}$, for  $i_1,\ldots,i_t,j_1,\ldots,j_t \in [n]$. 
\end{thm}
This stability version seems significantly weaker than the stability version in the cross $t$-intersection problem~\cite{keller2024t}, that allows deducing the same assertion once $|F||G| > \frac{3}{4} \cdot (n-t)!^2$. However, it is tight up to replacing the term $100$ in the denominator by a different constant. Indeed, let $F',G'$ be obtained from $F=G=\{\sigma \in S_n: \sigma(i)=i, \forall i \in [t]\}$ by adding to $F$ a single permutation $\sigma'$ such that $\sigma'(i) \neq i$ for all $i \in [t]$ and removing from $G$ all permutations $\tau$ that agree with $\sigma'$ on exactly $t-1$ elements. $F',G'$ are clearly cross $(t-1)$-intersection free, are not contained in the same family of the form $\{\sigma \in S_n:\sigma(i_1)=j_1,\ldots,\sigma(i_t)=j_t\}$,  and a direct calculation shows that $|F'||G'| \geq (1-\frac{1}{(ct)^t})\cdot (n-t)!^2$ for some $c>0$. (Indeed, the probability that $\tau \in G$ intersects $\sigma'$ in exactly $t-1$ elements is roughly the probability that a permutation $\tau \in S_n$ has exactly $t-1$ fixed points, which is of order $\frac{1}{O(t)^t}$). This stark difference between the stability versions emphasizes the difference between the $t$-intersection problem and the forbidden $(t-1)$-intersection problem, even in the range where the extremal families in both problems are the same.

\medskip \noindent \textbf{The single-family case.} A special case of Theorem~\ref{thm:main} proves the conjecture of Ellis~\cite{Ell14} itself, for $t \leq c \frac{n}{\log n}$. Specifically, we obtain the following result, which includes an essentially optimal stability version.
\begin{cor}\label{cor:main}
    There exists $c>0$ such that the following holds. Let $n,t \in \mathbb{N}$, such that $t\leq c \frac{n}{\log{n}}$. Let $F \subset S_n$ such that for any $\sigma, \tau \in F$, $|\{i \in [n]:\sigma(i)=\tau(i)\}|\neq t-1$. Then $|F| \leq (n-t)!$. 
    
    Moreover, if 
    \[
    |F| >  \left(1-\frac{1}{2(100t)^t} \right)\cdot (n-t)!,
    \]
    then $F \subset \{\sigma \in S_n:\sigma(i_1)=j_1,\ldots,\sigma(i_t)=j_t\}$, for some $i_1,\ldots,i_t,j_1,\ldots,j_t \in [n]$.
\end{cor}
As was mentioned above, the best previous published result on Ellis' conjecture, by Kupavskii and Zakharov~\cite{KZ22}, proves it for $t \leq \tilde{O}(n^{1/3})$.

\subsection{Our techniques} Besides classical combinatorial techniques, our proof makes crucial use of two recently proposed techniques. The first is \emph{hypercontractivity for global functions}, developed by Keevash, Lifshitz, Long and Minzer~\cite{KLLM21}, and more specifically, its sharp variant for functions over the symmetric group, developed by Keevash and Lifshitz~\cite{keevash2023sharp}. The second is the \emph{spreadness} technique that was introduced in the breakthrough work of Alweiss, Lovett, Wu and Zhang~\cite{ALWZ21} on the Sunflower conjecture, and was enhanced and developed by Kupavskii and Zakharov~\cite{KZ22} into the aforementioned \emph{spread approximation method}. Interestingly, while hypercontractivity for global functions and spreadness seem complementary, our proof uses both of them together -- a spreadness argument is applied to families constructed using a hypercontractivity argument.\footnote{We note that both the spread approximation method and hypercontractivity for global functions were used in the recent work of Kupavskii and Noskov~\cite{KupavskiiN24} on the Duke-Erd\H{o}s problem. However, they were used separately, to solve different cases of the problem.}  

\subsection{Proof overview} 
\label{sec:sub:Proof-overview}
For presenting the stages of the proof, we need a few more definitions. A \emph{$k$-restriction} of a family $F \subset S_n$ is $F_{i_1 \to j_1,\ldots,i_k \to j_k}= \{\sigma \in F: \sigma(i_1)=j_1 ,\ldots,\sigma(i_k)=j_k\}$. A restriction is called \emph{global} if no further restrictions increase the relative density of the family significantly. (A formal definition is given in Section~\ref{sec:Preliminaries}). For any $i,j \in [n]$, the family $(S_n)_{i \to j} = \{\sigma \in S_n: \sigma(i)=j\}$ is called a \emph{dictatorship}. For any $i_1,\ldots,i_t,j_1,\ldots,j_t \in [n]$, the family $(S_n)_{i_1 \to j_1,\ldots,i_t \to j_t} = \{\sigma \in S_n: \sigma(i_1)=j_1,\ldots,\sigma(i_t)=j_t\}$ is called a \emph{$t$-umvirate}. 

Our proof consists of two main steps:
\begin{enumerate}
    \item \emph{Cross $(t-1)$-intersection free families have a density bump inside a dictatorship.} We show that there exists $c>0$ such that if $t<c\frac{n}{\log n}$ and $F,G \subset S_n$ are cross $(t-1)$-intersection free, then there exist $i,j \in [n]$ such that the relative densities of both $F$ and $G$ inside the dictatorship $(S_n)_{i \to j}$ are significantly larger than the densities of $F$ and $G$ in $S_n$.

    \item \emph{Upgrading density bump inside a dictatorship into containment in a $t$-umvirate.} We show that the first step can be applied sequentially to deduce that $F,G$ are essentially contained in a $t$-umvirate $(S_n)_{i_1\to j_1,\ldots,i_t \to j_t}$ for some $i_1,\ldots,i_t$ and $j_1,\ldots,j_t$. 
\end{enumerate}

The second step is a technically involved inductive argument which requires incorporating a stability statement as part of the proof.

The first step is a bit shorter, but is the conceptually harder one. We assume to the contrary that $F,G$ do not have a density bump inside a dictatorship, and reach a contradiction. This is achieved in two sub-steps. 
\begin{enumerate}[label=(\alph*)]
    \item \emph{Using hypercontractivity to find global cross $1$-intersecting subfamilies of $F,G$.} We find sub-families $F' \subset F$, $G' \subset G$ that are global and cross $0$-intersection free (a.k.a.~cross $1$-intersecting). To this end, we first construct global restrictions $F',G'$ of $F,G$. Then, we use sharp hypercontractivity for global functions in $S_n$~\cite{keevash2023sharp} to show that there exist $i_1,j_1 \in [n]$ such that the relative densities of the restrictions $F'_{i_1 \to j_1},G'_{i_1 \to j_1}$ are \emph{not much smaller} than the densities of $F',G'$. These restrictions are clearly cross $(t-2)$-intersection free and are `not too small'. By repeating this argument (i.e., taking global restrictions of the two families and using hypercontractivity to find a common restriction that does not reduce the measure too much) iteratively $t-1$ times, we arrive at subfamilies of $F,G$ that are global and cross $0$-intersection free -- that is, cross $1$-intersecting.

    \item \emph{Using spreadness to obtain a contradiction.} We use the spreadness technique to show that two global families of permutations cannot be cross $1$-intersecting, thus obtaining a contradiction.
\end{enumerate}

\medskip \noindent \textbf{Related work.} The general strategy of our proof is similar to the proof of~\cite[Theorem~1]{keller2024t} which asserts that there exists $c>0$ such that for all $t<cn$, any two cross $t$-intersecting families $F,G \subset S_n$ satisfy $|F||G|\leq ((n-t)!)^2$. The conceptually easier second step of the proof is also generally similar to the corresponding step in~\cite{keller2024t} -- it is just technically harder. 

However, the conceptually harder proofs of the first step differ essentially completely. In~\cite{keller2024t}, the proof begins with a reduction to the setting of functions over the hypercube $\{0,1\}^{n^2}$ endowed with a biased measure, and the second sub-step of the proof utilizes a sharp threshold result obtained using sharp hypercontractive inequalities for functions over the hypercube~\cite{Omri}, the FKG inequality~\cite{FKG} and the biased version of the Ahlswede-Khachatrian complete intersection theorem~\cite{Filmus17} to obtain a contradiction. The entire proof works with $t$-intersecting families, and as a result, the first sub-step is relatively easy and does not use hypercontractivity. In our proof, the main goal of the first sub-step is to move from the non-monotone setting of cross $(t-1)$-intersection free families to the monotone setting of cross $1$-intersecting families. This is obtained by an iterative process in which we do not reduce the problem to the hypercube setting and instead, we make use of the special structure of $S_n$ by applying the recent result of Keevash and Lifshitz~\cite{keevash2023sharp} on sharp hypercontractivity for global functions over $S_n$. In the second sub-step, a central obstacle we have to overcome is that due to the $(t-1)$-step iterative process, the families we obtain are relatively small, unlike in~\cite{keller2024t}. As a result, hypercontractivity does not yield a sufficiently strong bound, and we have to use spreadness instead.\footnote{We note that it is possible to use a hypercontractivity argument in the second sub-step instead of the spreadness argument we use (see Remark at the end of Section~\ref{sec:spreadness}). However, this yields the assertion of Theorem~\ref{thm:main} only in the inferior range $t<c\frac{n}{(\log n)^2}$, for some $c>0$.}  

\subsection{Open problems} A main problem which remains for further research is, what is the range in which the conjecture of Ellis holds.
\begin{question}[Ellis and Lifshitz~\cite{DN22}]
    Find the maximum value $t(n)$ for which the maximum size of a $(t-1)$-intersection free family in $S_n$ is $(n-t)!$.  
\end{question}
For $t$-intersecting families in $S_n$, Ellis, Friedgut and Pilpel~\cite{EFP11} conjectured that the maximum size is $(n-t)!$ for all $t<\frac{n}{2}$, and this conjecture was proved by Kupavskii~\cite{Kup24a} for all sufficiently large $n$. It might be tempting to conjecture that the same holds for $(t-1)$-intersection free families in $S_n$. However, this is not the case. For any even integer $n$, the family 
\[
F=\{\sigma \in S_n: \forall k \leq \frac{n}{2} \;\exists \ell \leq \frac{n}{2}, \sigma(\{2k-1,2k\})=\{2\ell-1,2\ell\}\}
\]
consisting of all permutations that `keep all pairs $\{2k-1,2k\}$ unseparated' is clearly $(t-1)$-intersection free for every even $t$. We have $|F|=(\frac{n}{2})! \cdot 2^{n/2}$, which for every $\epsilon>0$, is larger than $(n-t)!$ for all $t>\frac{n}{2}(1-\frac{1-\epsilon}{\log_2(n/2)})$, provided $n\geq n_0(\epsilon)$. 

Another example is the family $G \subset S_n$ of all permutations composed of $\frac{n}{2}$ cycles of length $2$. Like $F$, the family $G$ is $(t-1)$-intersection free for every even $t$. We have $|G|=(n-1)(n-3)\cdots=\frac{n!}{(n/2)! \cdot 2^{n/2}} =\frac{|F|}{\Theta(\sqrt{n})}$, and thus, for every $\epsilon>0$, $|G|>(n-t)!$ for all $t>\frac{n}{2}(1-\frac{1-\epsilon}{\log_2(n/2)})$, provided $n\geq n_0(\epsilon)$. It will be interesting to find examples of $(t-1)$-intersection free families of permutations of size $>(n-t)!$ for some $t \leq \frac{n}{2}(1-\omega(\frac{1}{\log n}))$, if at all such examples exist.


\section{Preliminaries}
\label{sec:Preliminaries}

In this section we present definitions, notations and previous results that will be used in the sequel.

\subsection{Definitions and notations}

For convenience, we denote the families of permutations we consider by $A,B \subset S_n$. The indicator function of $A$ is $1_A:S_n \to \{0,1\}$, defined by $1_A(\sigma)=1$ if $\sigma \in A$ and $1_A(\sigma)=0$ otherwise. In places where the notation $1_A$ becomes too cumbersome due to the use of many subscripts and superscripts, we use $1\{A\}$ instead. For two permutations $\sigma,\tau \in S_n$, we use the notation $\sigma \cap \tau = \{i \in [n]:\sigma(i)=\tau(i)\}$.

Throughout the proofs, we use several universal constants. We denote them by $c,c_0,c_1,\ldots$, and we state dependencies between them when those dependencies are important. We didn't try to get optimal constant factors in our results. In several places in the sequel, we use the uniform measure on $S_n$ and on other spaces. Where the ambient space is clear from the context, we denote the uniform measure on it by $\mu$.

\medskip \noindent \textbf{The biased measure on the hypercube.} We also consider the biased measure $\mu_p$ on the hypercube $\{0,1\}^N$, defined by $\mu_p(x)=p^{\sum x_i}(1-p)^{N-\sum x_i}$. A subset $W \subset [N]$ is called \emph{$p$-random} if each $i \in [N]$ is included in $W$ with probability $p$, independently of the other elements. The family $2^{[N]}$ of all subsets of $[N]$ naturally corresponds to $\{0,1\}^N$. Under this correspondence, $p$-random subsets of $N$ correspond to elements of $\{0,1\}^N$  distributed according to $\mu_p$.

\subsection{Spreadness} For $r>1$, a probability measure $\mu_0$ on  $2^{[n]}$ is called \emph{$r$-spread} if
\[
\mu(\{F : X \subseteq F\}) \leq r^{-|X|}
\]
for any $X \subset [n]$. A family $A \subseteq 2^{[n]}$ is called \emph{$r$-spread} if the probability measure defined by its normalized indicator is $r$-spread.

The following result, called the \emph{iterated refinement inequality}, is a sharpening due to Tao~\cite{TaoSunflower} of a result proved by Alweiss et al.~\cite{alweiss2020improved} in their breakthrough work on the sunflower conjecture. The exact formulation we use is taken from the work of Kupavskii and Zakharov~\cite[Theorem~4]{KZ22} who attribute the result to Tao~\cite{TaoSunflower}.
\begin{thm}[\cite{TaoSunflower}, Corollary~7]\label{thm:spread-families}
    Let \( n, r \geq 1 \), let \( \mu_0: 2^{[n]} \rightarrow [0,1] \) be an \( r \)-spread measure, and let \( W \) be an \( (m\delta) \)-random subset of \( [n] \). Then
\[
\Pr\left[\exists F \in \operatorname{supp}(\mu_0) : F \subset W \right] \geq 1 - \left(\frac{5 }{\log_2(r\delta)} \right)^m \cdot |\mu_0|,
\]
where $|\mu_0| := \sum_{S \subset 2^{[n]}} |S| \mu_0(S)$.
\end{thm}

\subsection{Restrictions and globalness} 

Recall that the \emph{dictator} $(S_n)_{i\to j}$ is the set of permutations that send $i$ to $j$. For $d$-tuples $I=(i_1,\ldots,i_d)$ and $J=(j_1,\ldots,j_d)$, the \emph{$d$-umvirate} $(S_n)_{I \to J}$ is the set of permutations that send $i_\ell$ to $j_{\ell}$ for all $\ell=1,\ldots,d$.  The restriction of a function $f$ to the $d$-umvirate $(S_n)_{I \to J}$ is denoted by $f_{I\to J}$ and is called a \emph{$d$-restriction}.

\medskip \noindent \textbf{Sub-permutation-spaces.} Throughout the proof, we shall treat extensively iterative processes of restrictions of functions over $S_n$. For this sake, we introduce the following notion:  
\begin{defn}
A \emph{sub-permutation-space} is the space obtained from $S_n$ by a restriction of some of the coordinates. We denote such a space by $P_{n,k}$, where the number of restricted coordinates is $k$, and it will be convenient for us to not specify which coordinates were restricted.

Two sub-permutation-spaces are called \emph{non-intersecting} if the restrictions that created them send the same coordinates to different coordinates (e.g., are of the form $i \to j_1$ and $i \to j_2$, where $j_1 \neq j_2$), and thus, cannot contain intersections. 
\end{defn}
We note that in the context of the degree decomposition discussed below,  
$P_{n,k}$ can be thought of as $S_{n-k}$. 

\medskip \noindent \textbf{Globalness.} For $\gamma>0$, a function $f:S_n \to \mathbb{C}$ (or $f:P_{n,k} \to \mathbb{C}$)  is  \emph{$\gamma$-global} if 
\[
\|f_{I\to J}\|_2 \le \gamma^{|I|}\|f\|_2
\]
for all restrictions $f_{I\to J}$. A family $A \subset S_n$ (or $A \subset P_{n,k}$) is $\gamma$-global if its indicator function $1_A$ is $\gamma$-global. Intuitively, this notion means that no fixing of coordinates increases the relative density of $f$ significantly.

For any family $A \subset S_n$ and any $\gamma>0$, we can construct a $\gamma$-global restriction of $A$ by choosing tuples $I,J$ of elements in $[n]$ such that $\frac{|A_{I \to J}|}{\gamma^{|I|}|(S_n)_{I \to J}|}$ is maximal over all the choices of $I,J \subset [n]$ with $|I|=|J|$ (and if there are several maxima, we pick one of them arbitrarily). It is easy to see that $A_{I \to J}$ is indeed $g$-global and satisfies  
\begin{equation}\label{Eq:g-global-forbidden2}
    \frac{|A_{I \to J}|}{|(S_n)_{I \to J}|} \geq \gamma^{|I|}\cdot \frac{|A|}{|S_n|},
\end{equation} 
a fact that we will use several times. A $\gamma$-global restriction of $A \subset P_{n,k}$ can be constructed similarly.

\subsection{Hypercontractivity for global functions over $S_n$}

A function $f\colon S_n\to \mathbb{C}$ is a $d$-junta if $\exists i_1,\dots,i_d$ such that $f(\sigma)$ depends only on  $\sigma(i_1),\dots,\sigma(i_d)$.
The \emph{level-$d$} space $V_{\leq d}\subseteq L_2(S_n)$ is the linear space generated by all $d$-juntas. For every $d \leq n$, we set $V_{=d}:= V_{\leq d}\cap (V_{\leq d-1})^\perp$,
and denote by $f^{=d}$ the projection of $f$ into $V_{=d}$. 

The \emph{degree decomposition} of a function $f\colon S_n\to \mathbb{C}$ is $f=\sum_d f^{=d}$. This decomposition, introduced by Ellis, Friedgut and Pilpel~\cite{EFP11}, is an analogue of the decomposition into levels of the Fourier expansion of functions over the hypercube $\{0,1\}^n$. We shall use two results pertaining to the degree decomposition.

The first is a lemma obtained by Keevash, Lifshitz and Minzer~\cite{keevash2024largest}, which gives a precise description of the first level in the decomposition. 
\begin{lem}[{\cite[ Lemmas 3.1, 3.2]{keevash2024largest}}]\label{lem:level-1} 
    For any $f:S_n\rightarrow \mathbb{R}$, we have 
    \[
    f^{=1} = \sum_{i,j}a_{ij}x_{i  \rightarrow j},
    \] 
    where $x_{i \to j} = 1\{\sigma \in S_n:\sigma(i)=j\}$ and   
    $a_{ij}=(1-\frac{1}{n})(\mathbb{E}[f_{i\rightarrow j}]-\mathbb{E}[f])$.
    Moreover, 
    \[
    \|f^{=1}\|_2^2 = \frac{1}{n-1} \sum a_{ij}^2.
    \]
\end{lem}

The second result, which is the `hypercontractivity' component in our proof, is the \emph{level-$d$ inequality} for functions over $S_n$. This inequality is the main technical result in the recent breakthrough work of Keevash and Lifshitz~\cite{keevash2023sharp} on hypercontractivity for global functions over $S_n$.
\begin{thm}[{\cite[Theorem 1.8]{keevash2023sharp}}]\label{theorem:level-d for global functions_intro}
There exists $C>0$, such that for any $n \in \mathbb{N}$ and for any $\gamma>1$, if $A \subseteq S_n$ is $\gamma$-global and $d\le \min(\tfrac{1}{8}\log(\frac{1}{\mu(A)}), 10^{-5}n)$, then 
\[\|1_A^{= d} \|_2^2\le \mu(A)^2 \left( C \gamma^4 d^{-1} \log \left(\frac{1}{\mu(A)}\right) \right)^d,\]
where $\mu$ is the uniform measure on $S_n$.
\end{thm}

\section{Global Families of Permutations Cannot be
Cross $1$-Intersecting} \label{sec:spreadness}

In this section we prove the spreadness lemma that will be used in the proof of Theorem~\ref{thm:main}.
\begin{prop}\label{lem:cross-intersecting-global}
        For any $\gamma>0$, there exist $n_0,c_0$, such that the following holds. Let $n \ge n_0$ and $i \le c_0n$, and let $A \subseteq S_n, B\subseteq S_{n-i}$ be families of permutations that are $\gamma$-global (inside $S_n, S_{n-i}$, respectively). Then $A,B$ are not cross $1$-intersecting. 
\end{prop}

\medskip \noindent \emph{Proof.}
      Assume to the contrary that $A \subseteq S_n, B\subseteq S_{n-i}$ are $\gamma$-global and cross $1$-intersecting. 
      First, we embed $S_{n-i}$ into $S_n$ by sending $\sigma \in S_{n-i}$ to $\sigma' \in S_n$ that agrees with $\sigma$ on $[n-i]$ and is the identity permutation on $[n] \setminus [n-i]$. Then, we embed $S_n$ into $U_n=[n]^n$ in the natural way (i.e., $\sigma$ goes to $(\sigma(1),\ldots,\sigma(n))$). Applying the first embedding to $B$ to obtain $B' \subset S_n$ and the second embedding to $A,B'$, we  obtain two cross $1$-intersecting families $\bar{A},\bar{B} \subset U_n$.

      We claim that $\bar{A}$ is $(\gamma e)$-global in $U_n$, and $\bar{B}$ is $(2 \gamma e)$-global in the projection on $[n]^{n-i}$. Indeed, note that the embedding preserves the structure of the sets and only multiplies the measure of all sets by the same value. The assertion we want to prove regarding $\bar{A}$ is that for each $d$-restriction $\bar{A}_{S\rightarrow x}$, we have:
      $$\frac{|\bar{A}|}{n^n} \cdot (\gamma e)^{d} \ge \frac{|\bar{A}_{S\rightarrow x}|}{n^{n-d}}.$$
      This is equivalent to
      $$\frac{|A|}{n!} \cdot (\gamma e)^{d} \ge \frac{|A_{S\rightarrow x}|}{(n-d)!} \cdot \frac{n^d}{n \dots (n-d+1)},$$
      which follows from the original globalness assumption as $e^d \ge \frac{n^d}{n \dots (n-d+1)}$. 
      By the same argument, we can see that the restriction of $\bar{B}$ to the first $n-i$ coordinates is $(\gamma e)$-global inside $[n-i]^{n-i}$. Adding the extra $i$ output options can cost only a factor of $\frac{n}{n-i} \le 2$ per restricted coordinate.
       
      As a second step, we embed $U_n$ into the space $((\{0,1\}^{n})^n,\mu_p) \sim (\{0,1\}^{n^2},\mu_p)$, by setting $E(i_1,\ldots,i_n)=(e_{i_1},\ldots,e_{i_n})$, where $e_{i_j}$ is the $i_j$'th unit vector in $\{0,1\}^n$. For convenience, we use the notation $N:=n^2$ hereafter.

      Note that the embedding preserves the intersection property. 
      In addition, since the image of the embedding is included in the `slice' $\{x \in \{0,1\}^{N}: \sum x_i = n\}$ on which the $\mu_p$ measure is uniform, the effect of the embedding on the measure of all sets is multiplication by the same factor. 

      Define $\hat{A}, \hat{B} \subset \{0,1\}^{N}$ by $\hat{A}=E(\bar{A})=\{E(x):x \in \bar{A}\}$ and $\hat{B}=E(\bar{B})$. Notice that $\hat{A}$ is $\frac{n}{\gamma e}$-spread in $\{0,1\}^{N}$, and that the restriction of $\hat{B}$ to the first $(N-i \cdot n)$ coordinates is $\frac{n}{2 \gamma e}$-spread in $\{0,1\}^{N-i\cdot n}$. In addition, each element in the family has exactly $n$ 1's among its $N$ coordinates. Hence, to complete the proof it is sufficient to prove the following lemma.
\begin{lem}
For any $\gamma>0$, there exist $n_0,c_0$, such that for each $n \ge n_0$ and each $i \le c_0n$, the following holds. Let $w \in \{0,1\}^{i \cdot n}$ be a constant vector with $i$ 1's. Let $\hat{A}\subset \{0,1\}^{N}$ be an $\frac{n}{\gamma e}$-spread family all of whose elements have exactly $n$ 1's, and let $\hat{B}\subset \{0,1\}^{N-i\cdot n}$ be an $\frac{n}{2\gamma e}$-spread family all of whose elements have exactly $n-i$ 1's.
Then the families $\hat{A}, \hat{B} \times w \subset \{0,1\}^N$ cannot be cross $1$-intersecting.    
\end{lem}
\begin{proof}
    The proof uses ideas from the proof of~\cite[Lemma~9]{KZ22}. Let $S$ be a $\frac{1}{2}$-random subset of $\{0,1\}^{N}$ and let $S^c$ be its complement. Taking a sufficiently large constant $C(\gamma)$ and applying Theorem~\ref{thm:spread-families} with the parameters $\delta = \frac{C(\gamma)}{n}$, $m = \frac{n}{2C(\gamma)}$, and the measure $\mu_0$ defined by $\mu_0(x) = \frac{1\{\hat{A}\}}{\mu(\hat{A})}$, we get that with probability at least $1-n \cdot (C_0(\gamma))^n$, there is $v \in A$ such that $v \subseteq S$ (where $C_0(\gamma)$ depends only on $\gamma$). Applying the same argument to $\hat{B}$, we obtain that with probability at least $(1-(n-i) \cdot (C_1(\gamma))^n) \cdot 2^{-i}$ (where $C_1(\gamma)$ depends only on $\gamma$), there exists $u \in \hat{B} \times w$ such that $u \subseteq S^c$. Here, the extra factor $2^{-i}$ is the probability that $S^c$ contains the $i$ 1's of $w$. As $2^{-i} \geq 2^{-c_0n}$, for a sufficiently small $c_0$ the sum of the probabilities is greater than 1. Hence, there is a choice of $S$ for which the events $v \subset S$ for some $v \in \hat{A}$ and $u \subset S^c$ for some $u \in \hat{B} \times w$ occur at the same time, contradicting the assumption that $\hat{A},\hat{B} \times w$ are cross $1$-intersecting. This completes the proof of the lemma and of the proposition.
\end{proof} 

\noindent \textbf{Remark.} We note that for the proof of Theorem~\ref{thm:main}, we do not need the full strength of Proposition~\ref{lem:cross-intersecting-global}; it is sufficient to show that cross $1$-intersecting global families of permutations $A,B$ must satisfy $|A||B| \ll ((n-t)!)^2$. One may try to achieve this using the global hypercontractivity technique, and specifically, using~\cite[Proposition~5]{keller2024t}, after embedding $A,B$ into $(\{0,1\}^{n^2},\mu_p)$ for an appropriate value of $p$. It turns out that this approach works only for $t = O(\frac{n}{\log^2 n})$. Indeed, if one uses the embedding described above (which leads to $p=\frac{1}{n}$), the loss in the measure during the transition makes the result of~\cite[Proposition~5]{keller2024t} too weak for implying meaningful results for the initial families. If instead, one applies the more complex embedding used in~\cite[Section 4.3]{keller2024t} which leads to $p=\frac{\log n}{n}$, then there is no measure loss in the transition, but the application of~\cite[Proposition~5]{keller2024t} yields only a bound of $e^{-c\frac{n}{\log n}}$ on $|A||B|$. This value is $\ll ((n-t)!)^2$ only for $t=O(\frac{n}{\log^2 n})$. Since we want to cover the entire range $t=O(\frac{n}{\log n})$, the technique of~\cite[Proposition~5]{keller2024t} is not sufficient, and we use the spreadness argument described above instead. 
\section{Density Bump Inside a Dictatorship}

In this section we essentially show that if two `large' families $A,B \subset S_n$ are cross $(t-1)$-intersection free then there must exist a restriction $i 
\to j$ such that both $\frac{|A_{i \to j}|}{|A|}\gg \frac{1}{n}$ and $\frac{|B_{i \to j}|}{|B|}\gg \frac{1}{n}$. Formally, the result is stated not for pairs of families in $S_n$, but rather for families $A,B$ that reside in non-intersecting sub-permutation spaces $P_{n,k}^1,P_{n,k}^2$. This somewhat more cumbersome statement (which has almost no effect on the proof) is needed, since this is the setting in which we will apply this assertion in the proof of Theorem~\ref{thm:main}.

\begin{prop}\label{lem:non-densitybump}    
For every $m>0$, there exist $n_0 \in \mathbb{N}$ and $c_1,c_2>0$ such that for every $n>n_0$ and $t \le  c_1 n$, the following holds. Let $k \leq \frac{n}{100}$ and let $P_{n,k}^1,P_{n,k}^2$ be non-intersecting sub-permutation spaces of $S_n$ (i.e., two spaces obtained from $S_n$ by non-intersecting restrictions of the same set $T$ of $k$ coordinates). Let $A\subset P^1_{n,k}, B\subset P^2_{n,k}$ be cross$(t-1)$-intersection free families. 
Assume $A,B$ do not have a density bump inside a dictatorship that is not on the restricted coordinates, meaning there are no $i_1 \in [n]\setminus T,j_1 \in [n]$, 
such that $|A_{i_1\rightarrow j_1}|>|A|\cdot \frac{m}{n}$ and $|B_{i_1\rightarrow j_1}|>|B|\cdot \frac{m}{n}$. Then 
\[
\mu(A) \cdot \mu(B) \le  e^{-c_2 \cdot n},
\]
where $\mu(A),\mu(B)$ are taken w.r.t.~the uniform measure in $P_{n,k}^1,P_{n,k}^2$, respectively. 
\end{prop}



As was described in Section~\ref{sec:sub:Proof-overview}, to prove the assertion, we assume to the contrary that for some $A,B$ such a restriction does not exist. We construct global restrictions of $A,B$, and then we use an iterative process to gradually transform the families into cross $1$-intersecting families without reducing their measure too much and while preserving globalness. Then, we apply Proposition~\ref{lem:cross-intersecting-global} to the resulting families to get a contradiction. 



A crucial technical element of the proof is managing the restricted coordinates. At each of the $t-1$ steps, we apply restrictions on both functions in order to make them global. The restrictions may be in different sets of coordinates, and we must make sure that they do not create unintended intersections between elements of the families. Furthermore, we have to make sure that the `common restrictions' made during the process are not in coordinates in which one of the functions was restricted separately before. To achieve this, we formulate the core lemma below not for families in $S_n$, but rather for families in different sub-permutation spaces $P^1_{n,k},P^2_{n,k'}$, which captures the fact that the families we consider at a certain step are a result of different restrictions made at the previous steps. In addition, we have to make sure that the total number of restricted coordinates is not too large. This is not obvious at all, as the number of steps (i.e., $t-1$) can be as large as $c_1n$. We prove that nevertheless, the total number of restricted coordinates (in all steps) never exceeds $c_6n$, for a small constant $c_6$ that depends only on $c_1,c_2$. 


\subsection{The core lemma}

Our core lemma states the following.
\begin{lem}\label{lem:good_restriction}
    There exist $n_0 \in \mathbb{N}$ and $c_3>0$ such that for every $n \geq n_0$, the following holds. Let $k \leq \frac{n}{100}$ and let $A' \subset P^1_{n,k}$, $B' \subset P_{n,k}^2$, where  $P_{n,k}^1,P_{n,k}^2$ are non-intersecting sub-permutation spaces of $S_n$.
    Let $A = A'_{I \to J}$, $B = B'_{I' \to J'}$, where $|I|,|I'|\leq \frac{n}{100}$, be subsets of non-cross-intersecting $8$-global restrictions of $A'$ and $B'$, respectively (where `non-cross-intersecting' means that there is no single-coordinate restriction $i' \rightarrow j'$ performed on both families).
    If $\mu(A) \cdot \mu(B) \ge e^{-c_3n}$, 
    then there exist non-restricted coordinates $i,j$, such that 
    \[\mathbb{E}[1\{A_{i\rightarrow j}\}] \ge \frac{\mathbb{E}[1_A]}{2}, \qquad \mathbb{E}[1\{B_{i\rightarrow j}\}]\ge \frac{\mathbb{E}[1_B]}{2}.
    \]
\end{lem}
\begin{proof}
     The proof of the lemma uses hypercontractivity for global functions over $S_n$. Let $A,B$ be families that satisfy the assumptions, for a sufficiently large value $n \geq n_0$ and a sufficiently small value $c_3$ (the exact requirement will be specified below).
     Let $T_A, T'_A$ be the sets of non-restricted input and output coordinates of $A$ and denote $n_A=|T_A|=|T'_A|$. By assumption, $n_A\geq n-\frac{n}{100}-\frac{n}{100}\geq \frac{49}{50}n$. Note that we may identify $A$ with a subset of $S_{n_A}$ (by changing names of the input and output possible values). Let $g:S_{n_A}\to\{0,1\}$ be the indicator of $A$, and for all $1\leq i,j \leq n_A$, let $g_{i \to j}$ be the indicator of $A_{i \to j}$. 
     
     Define the random variable $X:[n_A]^2 \to \mathbb{R}$ by $X(i,j)=\mathbb{E}[g_{i\to j}]-\mathbb{E}[g]$. It is clear that $\mathbb{E}[X]=0$, and hence we have
     \[
     \mathrm{Var}[X]=\mathbb{E}[X^2]=\frac{1}{n_A^2}\sum_{i,j}(\mathbb{E}[g_{i\to j}]-\mathbb{E}[g])^2.
     \]
     Recall that by Lemma~\ref{lem:level-1}, the first level of the degree decomposition of $g$ is equal to $\sum_{i,j}a_{ij}x_{i \to j}$, where $a_{ij}=(1-\frac{1}{n_A})(\mathbb{E}[g_{i\rightarrow j}]-\mathbb{E}[g])$, and we have $\|g^{=1}\|_2^2 = \frac{1}{n_A-1} \sum a_{ij}^2$. Thus, 
     \[
     \mathrm{Var}[X] = \frac{1}{n_A^2} \sum_{i,j} \left(\frac{n_A}{n_A-1}\right)^2 a_{ij}^2 = \frac{1}{(n_A-1)^2}\sum_{i,j}a_{ij}^2=\frac{1}{n_A-1}\|g^{=1}\|_2^2.
     \]
     By assumption, $g$ is $8$-global, and thus, by the level-$d$ inequality for global functions over the symmetric group (i.e., Theorem~\ref{theorem:level-d for global functions_intro}), we have $\|g^{=1}\|_2^2 \leq C\mathbb{E}[g]^2 \log \left(\frac{1}{\mathbb{E}[g]}\right)$, where $C$ is an absolute constant. Since by assumption, $\mu(A) \cdot \mu(B) \ge e^{-c_3n}$, we have $\log \left(\frac{1}{\mathbb{E}[g]}\right) \leq c_3n$, and hence, $\|g^{=1}\|_2^2 \leq Cc_3n\mathbb{E}[g]^2$. It follows that  
     \[
     \mathrm{Var}[X] = \frac{1}{n_A-1}\|g^{=1}\|_2^2 \leq \frac{Cc_3n\mathbb{E}[g]^2}{n_A-1}\leq 2Cc_3 \mathbb{E}[g]^2.
     \]
     Taking $c_3$ sufficiently small so that $2Cc_3 \leq \frac{1}{10}$, we get $\mathrm{Var}[X] \leq \frac{1}{10}\mathbb{E}[g]^2$. By Chebyshev's inequality, it follows that
     \[
     \Pr \left(|X-\mathbb{E}[X]| \ge \frac{1}{2}\mathbb{E}[g]\right) \le \frac{1}{25}, 
     \]
     and hence, a $\frac{24}{25}$ fraction of the values of $X$ are greater than $\frac{\mathbb{E}[g]}{2}$ (recall that $E[X] = E[g]$).
        
     This means that for a $\frac{24}{25}$ fraction of the values of $i \in T_A$ ,$j \in T'_A$, we have $\mathbb{E}[1\{A_{i\to j}\}]\geq \frac{1}{2}\mathbb{E}[1_A]$. By the same argument, if $T_B,T'_B$ denote the sets of non-restricted input and output coordinates of $B$, then for a  $\frac{24}{25}$ fraction of the values of $i' \in T_B$, $j' \in T'_B$, we have $\mathbb{E}[1\{B_{i'\to j'}\}]\geq \frac{1}{2}\mathbb{E}[1_B]$. 
     
     Note that the sets $T_A,T_B$ share at least $\frac{97}{100}n$ coordinates, as each of them is obtained from the same set of $\geq \frac{99}{100}n$ coordinates, which are the non-restricted input coordinates of $A,B$, by removing at most $\frac{1}{100}n$ coordinates. Similarly, the sets $(T'_A,T'_B)$ share at least $\frac{96}{100}n$ coordinates, as each of them is obtained from $[n]$ by removing at most $\frac{2}{100}n$ coordinates -- at most $\frac{1}{100}n$ in the restriction that created $P_{n,k}^1,P_{n,k}^2$ and at most $\frac{1}{100}n$ in the restriction $S \to x$ or $S' \to x'$, respectively. Hence, there exist $i \in T_A \cap T_B$ and $j \in T'_A \cap T'_B$ such that $\mathbb{E}[1\{A_{i\to j}\}]\geq \frac{1}{2}\mathbb{E}[1_A]$ and $\mathbb{E}[1\{B_{i\to j}\}]\geq \frac{1}{2}\mathbb{E}[1_B]$. This completes the proof.
\end{proof}

\subsection{Proof of Proposition~\ref{lem:non-densitybump}}

Assume to the contrary that for some $A, B$, the assertion of the proposition fails. As was explained above, we would like to move from $A,B$ to families which are cross $0$-intersection free, i.e., are cross $1$- intersecting. We perform the following process $t-1$ times:
\begin{enumerate}
    \item We construct from $A$ and $B$ $8$-global families 
    $A'',B''$ in a way that does not create intersections. The way to do this, explained below, consists of performing restrictions and removing at most half of the elements.
    
    \item We use Lemma~\ref{lem:good_restriction} to find $i,j$ such that we have $\mathbb{E}[1\{A''_{i\to j}\}]\geq \frac{1}{2}\mathbb{E}[1_{A''}]$ and 
    $\mathbb{E}[1\{B''_{i\to j}\}]\geq \frac{1}{2}\mathbb{E}[1_{B''}]$, and set $\bar{A}=A''_{i\rightarrow j}$, $\bar{B}=B''_{i\rightarrow j}$. We then rename $\bar{A},\bar{B}$ to $A,B$.
\end{enumerate}
Notice that each of the two steps in each iteration reduces the measure of the families by a factor of at most $2$, and hence, the measures of the families during the entire process satisfy $\mu(A) \cdot \mu(B) \ge e^{-c_4 \cdot n}$, for some constant $c_4$ that depends only on $c_1,c_2$. The families we obtain at the end of the process are global and cross $1$-intersecting. This essentially yields a contradiction to Proposition~\ref{lem:cross-intersecting-global} that completes the proof. Now we present the proof in detail.

\medskip \noindent \textbf{Constructing non intersecting global restrictions in the $\ell$'th iteration.}
We describe the $\ell$'th iteration of the process and show how the required global restrictions can be found. We may assume that we start with two cross $(t-\ell)$-intersection free families $A\subset P^1_{n,k}, B\subset P^2_{n,k'}$ without intersection in the restricted coordinates (except for common intersections created in Step~2 of the previous iterations, if there are such). These are either the families we started the proof with, or the families we got after Step~2 of the $(\ell-1)$'th iteration. Notice that either way, they do not have a density bump inside a dictatorship. Indeed, since the families $A'',B''$ obtained at Step~1 are 8-global, the families $\bar{A},\bar{B}$ obtained at Step 2 (renamed to $A,B$ at the beginning of the following iteration) are $2 \cdot 8^2$-global, and hence they satisfy $\mathbb{E}[\bar{A}_{i\rightarrow j}] \le 128 \mathbb{E}[\bar{A}]$ and $\mathbb{E}[\bar{B}_{i\rightarrow j}] \le 128 \mathbb{E}[\bar{B}]$. Thus, we have \begin{equation}\label{eq:large-sets}
    \mu(A) \cdot \mu(B) \ge e^{-c_4 \cdot n},    
    \end{equation}
    and for any $i,j$, 
\begin{equation}\label{eq:density-bump}
    |A_{i\rightarrow j}|\leq |A|\cdot \frac{k}{n},  \qquad   |B_{i\rightarrow j}|\leq |B|\cdot \frac{k}{n},
    \end{equation}
    for $k=\max\{m,128\}$. 
    
    First, we construct a 4-global restriction $A' = A_{S\rightarrow x}$ of $A$ in the way described in Section~\ref{sec:Preliminaries}.\footnote{In order to accommodate several types of restrictions in the same proof, we denote the restriction by $A_{S \to x}$ where $S,x$ are tuples of elements of $[n]$, instead of the usual notation $A_{I \to J}$.} 
    By~\eqref{Eq:g-global-forbidden2}, we have
    \begin{equation} \label{eq:calc}
        1 \ge \mu^{S^c}(A_{S\rightarrow x}) \ge \mu(A)\cdot 4^{|S|} \ge 4^{|S|} \cdot e^{-c_4 \cdot n},   
    \end{equation}
    which implies (after taking logarithms)
\begin{equation}\label{eq:size-s}
        |S| \le c_4\cdot n.
    \end{equation}
    On the side of $B$, we look at $B' = B_{S\rightarrow x^c} = \{\sigma \in B: \sigma|_S\neq x\}$. Note that we do not restrict $B$ to $B'$, but rather we view $B'$ inside the original space. (We do not want to restrict $B$ to $B'$, since the space in which such a restriction resides is not isomorphic to $S_{n'}$; it is only isomorphic to a union of such spaces, and hence is inconvenient to work with). We have 
    \begin{equation}
    |B'| \ge |B| - \sum_{i \in S} |B_{i\rightarrow x_i}| \ge |B|\cdot \left(1-\frac{k}{n}|S| \right)>\frac{|B|}{2},    
    \end{equation} 
    assuming $c_4$ is sufficiently small, as a function of $k=\max\{128,m\}$. (This can be guaranteed by taking $c_1,c_2$ sufficiently small, as a function of $m$).
    Notice that $A', B'$ are cross $(t-\ell)$-intersection free, and that $A'$ is $4$-global in the restricted space $P^1_{n,k+|S|}$.
    
    We then construct a $4$-global restriction $B'' = B'_{S' \rightarrow x'}$ of $B'$, and by the same argument as in~\eqref{eq:calc}, we have
    \begin{equation*}
        1 \ge \mu^{S'^c}(B'_{S'\rightarrow x'}) \ge \mu(B')\cdot 4^{|S|} \ge 4^{|S|} \cdot \frac{e^{-c_4 \cdot n}}{2},  
    \end{equation*}
    which gives us, for a sufficiently large $n$, 
\begin{equation}\label{eq:size-s'}
        |S'| \le c_5\cdot n,
    \end{equation}
    where $c_5$ depends only on $c_1,c_2$.

    On the side of $A'$, we look at $A'' = A'_{S'\rightarrow x'^c}$ (inside the original space, as above). Since $A'$ is $4$-global, we have
    \begin{equation}
    |A''| \ge |A'| - \sum_{i \in S'} |A'_{i\rightarrow x'_i}| \ge |A'|\cdot \left(1-\frac{4}{n}|S'| \right)>\frac{|A'|}{2}, 
    \end{equation}
    assuming $c_5$ is sufficiently small.
    Notice that $A'', B''$ are cross $(t-\ell)$-intersection free, that $A''$ is $8$-global in the space $P^1_{n,k+|S|}$ (since $A'$ is $4$-global in the same space and $|A''|\geq \frac{|A'|}{2}$), and that $B''$ is $4$-global in $P_{n,k'+|S'|}^2$.

\medskip \noindent \textbf{Applying the core lemma in the $\ell$'th iteration.} At this stage, we would like to apply Lemma~\ref{lem:good_restriction} to the $8$-global families $A'',B''$. For this, we have to prove that the number of coordinates that were restricted up to this stage is at most $\frac{n}{100}$. 

To show this, denote by $k_i$ the number of coordinates of $A$ which we restrict at Step~1 of the $i$'th iteration (for all $i \le t$). The measure of the family $A$ after $\ell$ steps is at least 
\[e^{-c_2 \cdot n} \cdot 4^{\sum_{i=1}^\ell k_i} 4^{-\ell}.
\] 
Indeed, the first term is the initial measure, each time we take a global restriction in Step~1, the measure increases by a factor of at least $4^{k_i}$, and then we lose a factor of at most $2$ in the measure by removing part of the elements when passing from $A'$ to $A''$ and another factor of at most $2$ each time Step 2 is performed. Since the measure cannot exceed $1$, substituting $\ell \le t \le c_1 n$ we get 
\[
\sum_{i=1}^\ell k_i \le c_6 n,
\]
where $c_6$ is a constant that depends only on $c_1,c_2$. Thus, for sufficiently small $c_1,c_2$, the number of restricted coordinates is indeed at most $\frac{n}{100}$.

Therefore, the families $A'',B''$ satisfy the assumptions of Lemma~\ref{lem:good_restriction}. By the lemma, there exist non-restricted coordinates $i,j$, such that $\mathbb{E}[1\{A''_{i\rightarrow j}\}] \ge \frac{\mathbb{E}[1_{A''}]}{2}$ and $\mathbb{E}[1\{B''_{i\rightarrow j}\}]\ge \frac{\mathbb{E}[B'']}{2}$. We define $\bar{A}=A''_{i \to j}$ and $\bar{B}=B''_{i \to j}$. Note that $\bar{A},\bar{B}$ are $128$-global. We rename these families to $A,B$ and begin with them the $(\ell+1)$'th iteration.

\medskip \noindent \textbf{Applying Proposition~\ref{lem:cross-intersecting-global} and completing the proof.}
At the end of the $t-1$ iterations, we have two $128$-global cross $0$-intersection free families $\bar{A},\bar{B}$, that satisfy $\mu(\bar{A}) \cdot \mu(\bar{B}) \ge e^{-c_7 \cdot n}$, where $c_7$ depends only on $c_1,c_2,m$.
We would like to apply to these families Proposition~\ref{lem:cross-intersecting-global}.  
However, these families reside in different spaces:
$\bar{A}$ resides in $(S_n)_{K\rightarrow K_1', S \rightarrow x, T\rightarrow T'}$ and $\bar{B}$ resides in $(S_n)_{K\rightarrow K_2', S' \rightarrow x', T\rightarrow T'}$, where $K \to K_1'$ and $K \to K_2'$ denote the initial restriction from $S_n$ to $P^1_{n,k},P^2_{n,k}$ we began with, $S \to x$ and $S' \to x'$ denote the non-intersecting restrictions we performed in Steps~1 of the iterations, and $T \to T'$ denotes the common restrictions we performed in Steps~2 of the iterations. 
We would like to move the families into the same space.

Assume w.l.o.g.~that $|S| \le |S'|$.  We replace $\bar{B}$ by 
a family $\bar{B}'$, by performing a shifting of the coordinates in $\bar{B}$ (i.e., calling them by different names, or conjugating by a permutation), such that 
\[K \cup S' \rightarrow K \cup S \cup R, \qquad K_2'\cup x' \rightarrow K_1' \cup x \cup R,
\]
for some set $R$. Note that we may choose $c_1,c_2$ to be sufficiently small so that  $r = |R| \le c_0 \cdot \frac{n}{2}$, where $c_0$ is the constant from Proposition~\ref{lem:cross-intersecting-global}. 
 
The transformation from $\bar{B}$ to $\bar{B}'$ does not decrease the intersection size between any element of $\bar{A}$ and any element of $\bar{B}$, and thus, the resulting families $\bar{A}$ and $\bar{B}'$ are cross-intersecting. The family $\bar{A}$ resides in a space isomorphic to $S_{n-t-|K|-|S|}$, where $n-t-|K|-|S| \ge n/2$, and the family $\bar{B}'$ resides in a subspace of it, isomorphic to $S_{n-t-|K|-|S|-r}$. The families remain $128$-global (each in its space), since renaming coordinates does not affect globalness. 
Hence, Proposition~\ref{lem:cross-intersecting-global} (applied to $\bar{A},\bar{B}'$) yields a contradiction. 

\section{Proof of the Main Theorem}

In this section we prove Theorem~\ref{thm:main}. Specifically, we prove by induction the following proposition, which includes the assertion of Theorem~\ref{thm:main} as the case $m=0$. 
\begin{prop} \label{Prop:main-forbidden}
There exists $\alpha>0$ such that the following holds for any $t \in \mathbb{N}$, for any $n>  \alpha t\log(t)$, and for any $0 \leq m \leq n-t$.
Let $P^1_{n,m}, P^2_{n,m}$ be non-intersecting sub-permutation spaces of $S_n$ and let $A\subset P^1_{n,m}, B\subset P^2_{n,m}$ be cross $(t-1)$-intersection free families. 
Then we have
$$|A||B| \le 4^{m}(n-m-t)!^2.$$ 
\end{prop}

\begin{proof}
We prove the theorem by a double induction: An outer induction on $t$ and an inner backward induction on $m$ for given values of $t,n$, starting with $m=n-t$ and ending with $m=0$.


\medskip \noindent \textbf{Induction basis.} The basis of the induction is all the cases ($n>\alpha(t), m\ge t\log(n)$). In these cases, the claim holds trivially, since we clearly have
\[
|A||B|\leq ((n-m)!)^2 \leq 4^{t\log n}((n-m-t)!)^2 \leq 4^m ((n-m-t)!)^2.
\]

\subsection{A core lemma} 

The following core lemma asserts that under the induction assumption, the assertion of Proposition~\ref{lem:non-densitybump} that any two large cross $(t-1)$-intersection free families have a density bump inside a dictatorship, can be strengthened to the claim that any two such families are \emph{almost included} in the same dictatorship. 
\begin{lem}\label{lemma:all-meassure-in-dictator}
There exists $\alpha>0$ such that the following holds.
Let $t \in \mathbb{N}$,  $n> \alpha t\log(t)$, and $m \leq t\log n$ be fixed. Assume that the statement of Proposition~\ref{Prop:main-forbidden} holds for $n,t,$ and all $m'>m$. 
Let $P^1_{n,m}, P^2_{n,m}$ be non-intersecting sub-permutation spaces of $S_n$ and let $A\subset P^1_{n,m}, B\subset P^2_{n,m}$ be cross $(t-1)$-intersection free families.
If
\begin{equation}\label{Eq:large-sets-41}
|A||B| \ge 4^m\cdot\frac{2}{3}\cdot (n-m-t)!^2,    
\end{equation} 
then there exist $i,j$, s.t 
\begin{equation}\label{Eq:Almost-inside-dictatorship}
    |A_{i \rightarrow j}| \ge |A| \left(1-\frac{48}{n}\right) \qquad \mbox{and} \qquad  |B_{i \rightarrow j}| \ge |B|\left(1-\frac{48}{n}\right). 
\end{equation}
\end{lem}
\begin{proof} 
For a sufficiently large $\alpha$, we have $m \leq \frac{n}{100}$, $t \leq c_1n$, and $\mu(A)\mu(B)\geq e^{-c_2n}$, where $c_1,c_2$ are the constants of Proposition~\ref{lem:non-densitybump}. (Note that the measures $\mu(A),\mu(B)$ are taken with respect to the spaces $P_{n,m}^1,P_{n,m}^2$ these families reside in. Hence, $\mu(A)\mu(B)=\frac{|A||B|}{(n-m)!^2}\geq e^{-c_2n}$). Hence, by Proposition~\ref{lem:non-densitybump}, at least one of the families $A,B$ has a density bump inside a dictatorship. Formally, we may assume w.l.o.g.~that $A$ satisfies 
\begin{equation}\label{Eq:a-ge-k}
|A_{1\rightarrow i}| = \frac{a\cdot|A|}{n}, \qquad \mbox{where} \qquad a \ge k,
\end{equation}
for a value of $k$ that we will choose later. (The constants $c_1,c_2$, and in turn, the constant $\alpha_0$, depend on the choice of $k$). The families $A_{1\rightarrow i}$ and $\cup_{j \neq i}(B_{1\rightarrow j})$ are cross $(t-1)$-intersection free, and the families $A_{1\rightarrow i}$ and $B_{1\rightarrow j}$ 
are cross $(t-1)$-intersection free as well.  
The latter families reside in $P_{n,m+1}^1$ and in $P_{n,m+1}^2$, respectively, and hence, we can apply to them the induction hypothesis, to get
\begin{equation}\label{Eq:in}
\frac{a}{n}|A||\cup_{j \neq i}(B_{1\rightarrow j})|=|A_{1\rightarrow i}||\cup_{j \neq i}(B_{1\rightarrow j})| \le (n-1) \cdot 4^{m+1}\cdot (n-t-m-1)!^2.
\end{equation}
(Note that we cannot apply the induction hypothesis to $A_{1\rightarrow i}$ and $\cup_{j \neq i}(B_{1\rightarrow j})$, since the family $\cup_{j \neq i}(B_{1\rightarrow j})$ does not reside in $P_{n,m+1}^2$).
By ~\eqref{Eq:in}, we have
\begin{align*}
\begin{split}
    |A|\cdot|B_{1\rightarrow i}|+&\frac{n}{a}\cdot4^{m+1}(n-1)\cdot(n-t-m-1)!^2 \\&\ge|A|\cdot (|\cup_{j \neq i}(B_{1\rightarrow j})|+|B_{1\rightarrow i}|) \ge |A|\cdot|B| \ge \frac{2}{3}\cdot 4^{m}(n-t-m)!^2.
\end{split}
\end{align*}
Hence, we get
\begin{equation}\label{Eq:B-density}
    |B_{1\rightarrow i}| \ge \frac{4^{m}(n-t-m-1)!^2\cdot(\frac{2}{3}\cdot(n-t-m)^2-\frac{4n(n-1)}{a})}{|A|}.
\end{equation}
By the same argument as above, we have 
\[|B_{1\rightarrow i}||\cup_{j \neq i}(A_{1\rightarrow j})| \le 4^{m+1}(n-1) \cdot (n-t-m-1)!^2.
\]
Combining with ~\eqref{Eq:B-density}, we obtain
\begin{align*}
    |\cup_{j \neq i}(A_{1\rightarrow j})| \le \frac{4^{m+1}(n-1) \cdot (n-t-m-1)!^2\cdot|A|}{4^{m}(n-t-m-1)!^2\cdot(\frac{2}{3}\cdot(n-t-m)^2-\frac{4n(n-1)}{a})},
\end{align*}
and hence, 
\[
\frac{|\cup_{j \neq i}(A_{1\rightarrow j})|}{|A|} \le \frac{4(n-1)}{\frac{2}{3}\cdot(n-t-m)^2-\frac{4n(n-1)}{a}},
\]
which is equivalent to
\[
\frac{a}{n} = \frac{|A_{1\rightarrow i}|}{|A|} = 1 -\frac{|\cup_{j \neq i}(A_{1\rightarrow j})|}{|A|} \ge 1 - \frac{4(n-1)}{\frac{2}{3}\cdot(n-t-m)^2-\frac{4n(n-1)}{a}}.
\]
We write this as a quadratic inequality in the variable $a$ and solve it to get
\begin{align*}
    a \ge \frac{(n-t-m)^2  + \sqrt{(n-t-m)^4 - 24(n-1)(n-t-m)^2}}{2\frac{(n-t-m)^2}{n}},
\end{align*}
or
\begin{align*}
      a \le \frac{(n-t-m)^2  - \sqrt{(n-t-m)^4 - 24(n-1)(n-t-m)^2}}{2\frac{(n-t-m)^2}{n}}.
\end{align*}
In the second option, as $m\leq t\log n$, we can take $\alpha$ to be sufficiently large so that $n-t-m \geq \frac{n-1}{2}$, and get
\begin{align*}
    a &\le \frac{n-t-\sqrt{(n-t-m)^2-24(n-1)}}{2\frac{n-t-m}{n}} =\frac{n}{2}-\frac{n}{2}\cdot \sqrt{1-24\frac{n-1}{(n-t-m)^2}} \\&\leq \frac{n}{2} \left(1-\sqrt{1-\frac{48}{n-1}}\right) \leq \frac{n}{2}\left(1-\left(1-\frac{48}{n-1}\right)\right) \le 48,
\end{align*}
 contradicting the assumption~\eqref{Eq:a-ge-k} for any $k>48$.

In the first option, by the same argument we get
\begin{align*}
    a &\ge \frac{n-t+\sqrt{(n-t-m)^2-24(n-m-1)}}{2\frac{n-t-m}{n}} \\&= \frac{n}{2}+\frac{n}{2}\cdot \sqrt{1-24\frac{n-1}{(n-t-m)^2}} \geq n-48.
\end{align*}
Therefore, we have 
\[\frac{|A_{1\rightarrow i}|}{|A|}\geq 1-\frac{48}{n}.
\]

The same holds also for  $B_{1\rightarrow i}$. Indeed, 
by combining~\eqref{Eq:large-sets-41} and~\eqref{Eq:in}, and using the inequality $n-t-m \geq \frac{n-1}{2}$, we get
\begin{align*}
    1 - \frac{|B_{1\rightarrow i}|}{|B|} &= \frac{|\cup_{j \neq i}(B_{1\rightarrow j})|}{|B|} \le \frac{(n-1)4^{m+1}(n-t-m-1)!^2}{ (1-\frac{48}{n})4^{m} \cdot\frac{2}{3}(n-t-m)!^2} \\ &=\frac{6n(n-1)}{(n-48)(n-t-m)^2} 
    \le \frac{12n}{(n-48)(n-1)} \leq \frac{48}{n},
\end{align*}
where the last inequality holds assuming $n\geq 100$. Therefore, we have 
\[
\frac{|B_{1\rightarrow i}|}{|B|} \geq 1-\frac{48}{n},
\]
as asserted. 
\end{proof}

\subsection{Induction step}  

The induction step is divided into two sub-steps:
\begin{enumerate}
    \item We prove that the statement holds for all $m \geq 1$. At this step, we want to show that the assertion holds for the triple $(t,n,m)$ where $m \geq 1$, under the assumption that it holds for all triples $(t',n',m')$ such that either $(1\leq t'<t) \wedge (m'\geq 1)$ or $(t'=t) \wedge (n'=n) \wedge (m'>m)$. 

    \item We prove that the statement holds for $m=0$. At this step, we want to show that the assertion holds for the triple $(t,n,0)$, under the assumption that it holds for all triples $(t',n',m')$ such that either $(1\leq t'<t)$ or $(t'=t) \wedge (n'=n) \wedge (m'>0)$. 
\end{enumerate}
It is clear that the combination of the two steps, together with the induction basis, proves the assertion. Throughout the proof, we take $\alpha$ to be sufficiently large, so that the core lemma holds for all $n \geq \alpha t \log(t)$.

\medskip \noindent \textbf{Step~1: The case $m \geq 1$.} Assume to the contrary that the assertion fails for $(t,n,m)$ -- i.e., that there exist cross $(t-1)$-intersection free families $A\subset P^1_{n,m}, B\subset P^2_{n,m}$, 
such that
$$|A||B| > 4^{m}(n-m-t)!^2.$$ 
The families $A,B$ satisfy the assumptions of Lemma~\ref{lemma:all-meassure-in-dictator}, and hence, by~\eqref{Eq:Almost-inside-dictatorship} we have $\min(\frac{|A_{1\rightarrow i}|}{|A|}, \frac{|B_{1\rightarrow i}|}{|B|}) \ge 1-\frac{48}{n}$ (where the fixation of the restriction to $1 \to i$ is w.l.o.g.).  

$A_{1\rightarrow i}$ and $B_{1\rightarrow i}$ are cross $(t-2)$-intersection free families that reside in spaces isomorphic to $P^1_{n-1,m}, P^2_{n-1,m}$. We would like to apply Lemma~\ref{lemma:all-meassure-in-dictator} to these restrictions. They indeed satisfy the assumption of the lemma, as by the induction assumption, the assertion of Proposition~\ref{Prop:main-forbidden} is satisfied for all triples of the form $(t-1,n,m')$ where $m'>m$ and as we have 
 \begin{equation*}
 |A_{1\rightarrow i}|\cdot |B_{1\rightarrow i}| > 4^{m}(n-t-m)!^2 \left(1-\frac{48}{n} \right)^2 \ge \frac{2}{3}\cdot 4^{m}(n-t-m)!^2,    
 \end{equation*}
where the last inequality holds assuming $n\geq 300$. Therefore, by Lemma~\ref{lemma:all-meassure-in-dictator}, the families  $A_{1\rightarrow i}, B_{1\rightarrow i}$ are almost contained in the same dictatorship. Since for a sufficiently large $\alpha$ and for all $n \geq \alpha t \log(t)$, we have $(1-\frac{48}{n})^t \ge 1-\frac{1}{1000} > \frac{2}{3}$, we can repeat this step $t$ times, and 
deduce that each of the families
$A,B$ has at least $1-\frac{1}{1000}$ of its measure inside the restriction $1\rightarrow i_1, 2\rightarrow i_2,...,t\rightarrow i_t$. In particular, we have 
\[
|A||B|\leq (n-t-m)!^2 \cdot \left(\frac{1000}{999}\right)^2 \leq 4^m (n-t-m)!^2,
\]
where the second inequality holds since $m \geq 1$. This completes the proof of the first step.

\medskip \noindent \textbf{Step~2: The case $m=0$.}  
Assume to the contrary that the assertion fails for some families $A,B$ with respect to the parameters $(t,n,0)$. Note that here, we have $A,B \subset S_n$ (since $m=0$). The above argument can be applied verbatim in this case as well, to deduce that each of the families
$A,B$ has at least $1-\frac{1}{1000}$ of its measure inside the $t$-umvirate $U:= \{\sigma \in S_n: \sigma(1)=i_1, \sigma(2)=i_2,\ldots,\sigma(t)=i_t\}$, and in particular, we have
\[|A||B|\leq (n-t)!^2 \cdot \left(\frac{1000}{999}\right)^2.
\]
This upper bound is however not sufficiently strong
and we shall obtain a stronger inequality by bounding the measure of $A,B$ outside of $U$. 

Note that if both $A,B$ are included in $U$, then $|A||B| \leq |U|^2 = (n-t)!^2$ and we are done. Hence, we may assume w.l.o.g.~that there exists a $t$-umvirate 
$U':= \{\sigma \in S_n: \sigma(1)=j_1, \sigma(2)=j_2,\ldots,\sigma(t)=j_t\}$, such that $A \cap U' \neq \emptyset$, where $|\{\ell:i_\ell=j_\ell\}|=t-r$ for some $0<r \leq t$. We shall need the following lemma.
\begin{lem}\label{lem:cross-intersects-one}
   Let $n \geq 20$, let $r\leq \frac{n}{6}$, and let $j\leq r$. Let $\sigma \in S_n$ and let $U= \{\tau \in S_n: \tau(i_1)=i'_1, \tau(i_2)=i'_2,\ldots,\tau(i_r)=i'_r\}$ be an $r$-umvirate that disagrees with $\sigma$ on all its $r$ coordinates. Then the number of permutations $\tau \in U$ such that $|\sigma \cap \tau|\neq j$ is at most 
   \[
   \left(1-\frac{1}{4\cdot 2^j \cdot j!}\right) \left(n-r \right)!.
   \]
\end{lem}

\begin{proof}
    We may assume w.l.o.g.~that $\sigma$ is the identity permutation, that $\{i_1,\ldots,i_r\}=[r]$, and that $U$ sends $[r]$ to $\{r+1,\ldots,2r\}$. We shall bound from below the probability that $\tau \in U$ intersects $\sigma$ on exactly $j$ coordinates, which we denote by $f(n,r,j)$. We claim that  
    \[
    f(n,r,j)=\frac{{n-2r \choose j} \cdot (f(n-j,r,0) \cdot (n-r-j)!)}{(n-r)!}. 
    \]
    Indeed, we have only ${n-2r \choose j}$ possible ways to choose the $j$ fixed points of $\tau$, since $r+1,r+2,\ldots,2r$ cannot be fixed points, being already used as images of elements in $[r]$. Once the fixed points are picked, they can be thrown out and the problem reduces to the setting $(n-j,r,0)$. Hence, the number of such permutations for each choice of the fixed points is $f(n-j,r,0) \cdot (n-r-j)!$, and consequently, the overall probability $f(n,r,j)$ is as stated.
    
    We now claim that $f(n',r,0) \geq \frac{1}{4}$ assuming $n'-r \geq 10$, an assumption that is satisfied for $n'=n-j$. Indeed, as for any $\tau \in U$, the elements in $\{r+1,\ldots,2r\}$ cannot be fixed points, we can assign the other elements and compute the probability that they don't have a fixed point. Hence, we can treat $\tau$ as an injection from $\{2r+1,\ldots,n'\}$ to $[r] \cup \{2r+1,\ldots,n'\}$, complete the injection to a permutation over the set $[r] \cup \{2r+1,\ldots,n'\}$ in a random way, and compute the probability that the resulting permutation does not have a fixed point inside the set $\{2r+1,\ldots,n'\}$. This probability is clearly lower bounded by the probability that a permutation over a set of $n'-r$ elements is a derangement, which is $\approx \frac{1}{e}$, and in particular, is at least $\frac{1}{4}$ if $n'-r \geq 10$. 

    Therefore, we have
    \[
    f(n,r,j)=\frac{{n-2r \choose j} \cdot (f(n-j,r,0) \cdot (n-r-j)!)}{(n-r)!} \geq \frac{(n-2r-j)^j}{4j!(n-r)^j} \geq \frac{1}{4\cdot 2^j \cdot j!}, 
    \]
    where the last inequality holds since $j \leq r \leq \frac{n}{6}$. This completes the proof of the lemma.
\end{proof}

Let $U':= \{\sigma \in S_n: \sigma(1)=j_1, \sigma(2)=j_2,\ldots,\sigma(t)=j_t\}$ be a $t$-umvirate distinct from $U$ such that $A \cap U' \neq \emptyset$. Denote $|\{\ell:i_\ell=j_\ell\}|=t-r$, for some $0<r \leq t$. 
It follows from the lemma, applied with $(n-(t-r),r,r-1)$ in place of $(n,r,j)$, that 
\[
|B \cap U| \leq \left(1-\frac{1}{4\cdot 2^{r-1} \cdot (r-1)!}\right)(n-t)!,
\]
as for each $\sigma \in B \cap U$ and $\tau \in A \cap U'$ we have $|\sigma \cap \tau|\neq t-1$. As $4 \cdot 2^{r-1} \cdot (r-1)! \leq 2 \cdot 2^r r^r$, we have
\begin{equation}\label{eq:temp0}
|B \cap U| \leq \left(1-\frac{1}{2 \cdot (2r)^r}\right)(n-t)!.
\end{equation}
We would like to bound $|A \cap U'|$ using the induction hypothesis. Note that the families 
$A \cap U' = \{\sigma \in A: \sigma(1)=j_1, \sigma(2)=j_2,\ldots,\sigma(t)=j_t\}$
and
$B \cap U = \{\sigma \in B: \sigma(1)=i_1, \sigma(2)=i_2,\ldots,\sigma(t)=i_t\}$
are cross $(r-1)$-intersection free. Hence, due to the induction hypothesis (which covers both the case $(t',n',m')$ where $t'<t$ that we obtain for $r<t$, and the case $(t',n',m')$ where $t'=t, n'=n, m'>m$ that we obtain for $r=t$), we can apply Proposition~\ref{Prop:main-forbidden} with $(t',n',m')=(r,n-(t-r),r)$ in place of $(t,n,m)$ to get that 
\begin{equation}\label{eq:temp1}
    |A \cap U'|\cdot|B \cap U| \le 4^{m'}\cdot (n'-m'-t')!^2 = 4^r(n-t-r)!^2. 
\end{equation}
Note that we have $|B \cap U| \geq \frac{1}{2}(n-t)!$, since otherwise, 
\[
|A||B| \leq |A\cap U| |B \cap U|\cdot \left(\frac{1000}{999}\right)^2 < (n-t)!^2.
\]
Therefore, 
\begin{equation}\label{eq:temp1.1}
    |A \cap U'| \le \frac{2 \cdot 4^r(n-t-r)!^2}{(n-t)!}. 
\end{equation}
Now, we would like to bound from above $\sum_{U'} |A \cap U'|$, where the sum is over all $t$-umvirates $U' \neq U$. For every fixed $r$, the number of $t$-umvirates $U' = 
\{\sigma \in S_n: \sigma(1)=i_1, \sigma(2)=i_2,\ldots,\sigma(t)=i_t\}$
such that $|\{\ell:i_\ell=j_\ell\}|=t-r$, is 
\[
{t \choose t-r} {n-t \choose r} \leq \frac{t^r}{r!} \cdot \frac{n^r}{r!} \leq \frac{t^r n^r e^{2r}}{r^{2r}}=\left(\frac{e^2 tn}{r^2}\right)^r.
\]
Hence, denoting 
\[
\mathcal{U}'_{t-r}= \{U' =\{\sigma \in S_n: \sigma(1)=j_1, \sigma(2)=j_2,\ldots,\sigma(t)=j_t\}: |\{\ell:i_\ell=j_\ell\}|=t-r\},
\]
we have, for all $0<r \leq t$,
\[
\sum_{U' \in \mathcal{U}_{t-r}} |A \cap U'| \leq (n-t)!\left(\frac{e^2 tn}{r^2}\right)^r \cdot \frac{2 \cdot 4^r(n-t-r)!^2}{(n-t)!^2} \leq (n-t)! \left(\frac{8e^2 tn}{(n-t-r)^2 r^2}\right)^r. 
\]
For a sufficiently large $\alpha$, we have $\frac{8e^2 tn}{(n-t-r)^2}\leq \frac{1}{8}$ for all $n \geq \alpha t \log(t)$ and all $r \leq t$. Hence, 
\[
\sum_{U' \in \mathcal{U}_{t-r}} |A \cap U'| \leq (n-t)! \left(\frac{8e^2 tn}{(n-t-r)^2 r^2}\right)^r \leq (n-t)! \left(\frac{1}{8r^2}\right)^r. 
\]
Summing over all possible values of $r$, and denoting by $r_1$ be the minimal $r$ for which there exists $U' \in \mathcal{U}_{t-r}$ such that $A \cap U' \neq \emptyset$, we obtain 
\[
|A \setminus (A\cap U)| = 
\sum_{r=r_1}^t \sum_{U' \in \mathcal{U}_{t-r}} |A \cap U'| \leq \sum_{r=r_1}^{t} (n-t)! \left(\frac{1}{8r^2}\right)^r \leq (n-t)! \cdot 2 \left(\frac{1}{8r_1^2}\right)^{r_1}.
\]
Combining with~\eqref{eq:temp0}, we get
\begin{equation}\label{eq:temp2}
    |A \setminus (A\cap U)| \leq 2 \left(\frac{1}{8r_1^2}\right)^{r_1}(n-t)!, \qquad |B\cap U| \leq \left(1-\frac{1}{2 \cdot (2r_1)^{r_1}}\right)(n-t)!. 
\end{equation}
To conclude the argument, we consider two cases.
\begin{enumerate}
    \item \textbf{Case~1: $B \subset U$.}
    As in this case we have $B=B\cap U$, and as clearly, $|A \cap U|\leq (n-t)!$,~\eqref{eq:temp2} yields
    \begin{align*}
    |A||B|&=(|A \cap U|+|A\setminus (A\cap U)|)(|B \cap U|) \\
    &\leq \left(1+2 \left(\frac{1}{8r_1^2}\right)^{r_1}\right)(n-t)! \left(1-\frac{1}{2 \cdot (2r_1)^{r_1}}\right)(n-t)! \\
    &\leq (n-t)!^2 \left(1+2 \left(\frac{1}{8r_1^2}\right)^{r_1} -\frac{1}{2 \cdot (2r_1)^{r_1}}\right) \leq (n-t)!^2,
    \end{align*}
    where the last inequality holds since for any $r_1 \geq 1$ we have $2(\frac{1}{8r_1^2})^{r_1}\leq \frac{1}{2 \cdot (2r_1)^{r_1}}$. 
    
    \item \textbf{Case~2: $B \not\subset U$.} Let $r_2>0$ be the minimal $r$ for which there exists $U' \in \mathcal{U}_{t-r}$ such that $B \cap U' \neq \emptyset$. By the same argument as above, with the roles of $A,B$ interchanged, we have 
\begin{equation}\label{eq:temp3}
    |B \setminus (B\cap U)| \leq 2 \left(\frac{1}{8r_2^2}\right)^{r_2}(n-t)!, \qquad |A\cap U| \leq \left(1-\frac{1}{2 \cdot (2r_2)^{r_2}}\right)(n-t)!. 
\end{equation}
Combining~\eqref{eq:temp2} and~\eqref{eq:temp3}, we obtain
\[
|A| = |A \cap U|+|A \setminus (A\cap U)|\le (n-t)!\cdot \left(1-\frac{1}{2 \cdot (2r_2)^{r_2}}+ 2 \left(\frac{1}{8r_1^2}\right)^{r_1}\right),
\]
and 
\[
|B| = |B \cap U|+|B \setminus (B\cap U)|\le (n-t)!\cdot \left(1-\frac{1}{2 \cdot (2r_1)^{r_1}}+ 2 \left(\frac{1}{8r_2^2}\right)^{r_2}\right).
\]
Hence,
\begin{equation*}
    |A||B| \le (n-t)!^2\cdot \left(1-\frac{1}{2 \cdot (2r_2)^{r_2}}+ 2 \left(\frac{1}{8r_1^2}\right)^{r_1}\right) \left(1-\frac{1}{2 \cdot (2r_1)^{r_1}}+ 2 \left(\frac{1}{8r_2^2}\right)^{r_2}\right).
\end{equation*}
Note that if $x=(1+\beta_1)(1+\beta_2)$ where $\beta_1,\beta_2\geq -2$ and $\beta_1+\beta_2 \leq 0$ then $x \leq 1$. Hence, in order to prove that $|A||B| \leq (n-t)!^2$, it is sufficient to show that 
\[
-\frac{1}{2 \cdot (2r_2)^{r_2}}+ 2 \left(\frac{1}{8r_1^2}\right)^{r_1}-\frac{1}{2 \cdot (2r_1)^{r_1}}+ 2 \left(\frac{1}{8r_2^2}\right)^{r_2} \leq 0.
\]
This indeed holds, as for any $r \geq 1$ we have $2(\frac{1}{8r^2})^{r}\leq \frac{1}{2 \cdot (2r)^{r}}$, and we can use this for both $r_1$ and $r_2$.   
\end{enumerate}
This completes the inductive proof of Proposition~\ref{Prop:main-forbidden}, and thus, of Theorem~\ref{thm:main}.
\end{proof}
The stability version (namely, Theorem~\ref{thm:stability}) follows from the same proof, with a careful analysis of Cases~1,2 at the end of the proof.

\bibliographystyle{plain}
\bibliography{refs}

\end{document}